\documentclass[a4paper,12pt]{amsart}

\usepackage{amsmath}
\usepackage{amssymb}
\usepackage{ascmac}
\usepackage{amsthm}
\usepackage{color}
\usepackage{comment}
\usepackage{empheq}
\usepackage[dvipdfmx]{graphicx}
\usepackage[abbrev]{amsrefs}
\usepackage{enumitem}

\makeatletter

\@addtoreset{equation}{section}
\makeatother 

\theoremstyle{plain}
\newtheorem{thm}{Theorem}[section]
\newtheorem*{thm*}{Theorem}
\newtheorem{prop}[thm]{Proposition}
\newtheorem{lem}[thm]{Lemma}

\theoremstyle{definition}

\theoremstyle{remark}
\newtheorem{rem}[thm]{Remark}

\newcommand{\Ric}{\operatorname{Ric}}

\newcommand{\Hess}{\operatorname{Hess}}

\newcommand{\tr}{\operatorname{tr}}

\newcommand{\second}{{\rm I\hspace{-.01em}I}}

\newcommand{\Ind}{\operatorname{Ind}}

\newcommand{\eps}{\varepsilon}

\newcommand{\g}{\mathsf{g}}
\newcommand{\SL}{\mathsf{L}}
\newcommand{\D}{\mathsf{D}}

\newcommand{\e}{\mathrm{e}}
\renewcommand{\d}{\mathrm{d}}

\usepackage{latexsym}

\makeatletter
\@namedef{subjclassname@2020}{%
  \textup{2020} Mathematics Subject Classification}
\makeatother

\title[Stability of weighted minimal hypersurfaces]{Stability of weighted minimal hypersurfaces\\ under a lower $1$-weighted Ricci curvature bound}

\author{Yasuaki Fujitani}
\address{Graduate School of Mathematical Sciences, The University of Tokyo, 3-8-1 Komaba, Tokyo, 153-8914, Japan}
\email{yasuakifujitani@g.ecc.u-tokyo.ac.jp}

\author{Yohei Sakurai}
\address{Department of Mathematics, Saitama University, 255 Shimo-Okubo, Sakura-ku, Saitama-City, Saitama, 338-8570, Japan}
\email{ysakurai@rimath.saitama-u.ac.jp}

\date{February 6, 2026}
\begin{document}
\maketitle
\begin{abstract}
We will study the $1$-weighted Ricci curvature in view of the extrinsic geometric analysis.
We derive several geometric consequences concerning stable weighted minimal hypersurfaces in weighted manifolds under a lower $1$-weighted Ricci curvature bound.
We prove a Schoen-Yau type criterion, and conclude a structure theorem for three-dimensional weighted manifolds of non-negative $1$-weighted Ricci curvature.
We also show non-existence results under volume growth conditions, and conclude smooth compactness theorems.
\end{abstract}

\section{Introduction}
Let $(M,g,f)$ be an $n$-dimensional weighted manifold,
namely,
$(M,g)$ is an $n$-dimensional Riemannian manifold, and $f\in C^{\infty}(M)$,
which is endowed with the \textit{weighted measure}
\begin{equation}\label{eq:volume}
m_f := \e^{-f} \,v_g.
\end{equation}
For $N\in (-\infty,+\infty]$,
the \textit{$N$-weighted Ricci curvature} is defined as follows:
\begin{equation*}
\Ric_f^N:=\Ric + \Hess f-\frac{\d f\otimes \d f}{N-n}.
\end{equation*}
For $N = +\infty$, 
we interpret the last term as the limit $0$, 
and for $N = n$, 
we only consider a constant function $f$ and set $\Ric_f^n := \Ric$.
The parameter $N$ has been chosen from $[n,+\infty]$ in the classical metric measure geometry,
but in the last decade,
there have been progresses for the understanding of its role in $(-\infty,n)$ by Ohta \cite{O}, Klartag \cite{K}, Kolesnikov-Milman \cite{KM1}, \cite{KM2}, Milman \cite{M}, Wylie \cite{W2} and so on.
It enjoys the following monotonicity with respect to $N$:
For $N_1,N_2 \in [n,+\infty]$ with $N_1\leq N_2$,
and for $N_3,N_4 \in (-\infty,n)$ with $N_3\leq N_4$,
\begin{equation}\label{eq:monotonicity}
\Ric_f^{N_1}\leq  \Ric_f^{N_2} \leq \Ric_f^{\infty} \leq  \Ric_f^{N_3} \leq \Ric_f^{N_4}.
\end{equation}

In the present article,
we will focus on the case of $N=1$.
The $1$-weighted Ricci curvature has been investigated from various perspectives including comparison geometry (see e.g., \cite{W2}, \cite{WY}, \cite{Sa1}),
weighted sectional curvature (see e.g., \cite{KWY}),
Bochner and Reilly formulas (see e.g., \cite{LX}), 
substatic geometry (see e.g., \cite{BF}, \cite{CFMR}),
and optimal transport theory (\cite{Sa2}).
We develop the theory of $1$-weighted Ricci curvature in view of extrinsic geometric analysis.
Our research object is \textit{weighted minimal hypersurface} (also called \textit{$f$-minimal hypersurface}),
which is a critical point of the weighted volume functional associated with \eqref{eq:volume}.
We conclude several geometric consequences concerning stable $f$-minimal hypersurfaces in weighted manifolds under a lower $1$-weighted Ricci curvature bound.

\subsection{Lorentzian background}\label{subsec:Lorentzian}
We first recall the role of the $1$-weighted Ricci curvature in the literature of Lorentzian geometry.

The development of Lorentzian geometry plays an essential role in better understanding of general relativity;
especially, singularity and black hole.
One of the goals of global Lorentzian geometry is to conclude the geometric structure of Lorentzian manifolds under the \textit{Einstein equation} together with an energy condition.
On a Lorentzian manifold $(L,\mathfrak{g})$,
the Einstein equation is given by
\begin{equation*}
\Ric_{\mathfrak{g}}+\left( \Lambda-\frac{1}{2}R_{\mathfrak{g}}  \right)\mathfrak{g}=8\pi \mathsf{T},
\end{equation*}
where $\mathsf{T}$ is the stress-energy tensor,
and $\Lambda$ is the cosmological constant.
Furthermore,
energy conditions are non-negativity assumptions for the stress-energy tensor;
for instance,
the weak energy condition (WEC) says that the stress-energy tensor is non-negative for every timelike vector,
which is believed in general relativity that it holds for all physically reasonable matters.
Besides WEC,
the strong energy condition (SEC), the null energy condition (NEC), the dominant energy condition (DEC) are also well-known.
The classical structure results such as the Hawking-Penrose singularity theorem (\cite{HP}) and the Lorentzian splitting theorem (\cite{G1}, \cite{N}) are available under SEC and NEC.

A Lorentzian manifold $(L,\mathfrak{g})$ is called a \textit{static model} if it can be written in the form of
\begin{equation}\label{eq:substatic}
L=\mathbb{R}\times M,\quad \mathfrak{g}=-V^2\,dt^2+\mathsf{g}
\end{equation}
for some Riemannian manifold $(M,\mathsf{g})$ and positive function $V\in C^{\infty}(M)$.
Under the Einstein equation,
a static model $(L,\mathfrak{g})$ in the form of \eqref{eq:substatic} satisfies NEC (i.e., the stress-energy tensor is non-negative for every null vector) if and only if a triple $(M,\mathsf{g},V)$ satisfies the following \textit{substatic condition} (see \cite[Lemma 3.8]{WWZ}, \cite[Appendix A.1]{BF}):
\begin{equation}\label{eq:substatic2}
V \Ric_{\mathsf{g}}-\Hess_{\mathsf{g}} V+(\Delta_\mathsf{g} V) \mathsf{g}\geq 0.
\end{equation}
The timeslice of the de Sitter-Schwarzschild and Reissner-Nordstr\"{o}m space-times are known to be regarded as substatic triples (\cite[Proposition 2.1 and Section 5]{B}).
Borghini-Fogagnolo \cite{BF} have pointed out that
the following are equivalent (see \cite[Appendix A.2]{BF}):
\begin{itemize}\setlength{\itemsep}{+1.0mm}
\item A triple $(M,\mathsf{g},V)$ is substatic;
\item a weighted manifold $(M,g,f)$ given by
\begin{equation*}
f:=-(n-1)\log V,\quad g:=\e^{\frac{2f}{n-1}}\mathsf{g}
\end{equation*}
has non-negative $1$-weighted Ricci curvature.
\end{itemize}
Hence,
the study of weighted manifolds of non-negative $1$-weighted Ricci curvature is equivalent to that of substatic triples via conformal change of metric.

\subsection{Lorentzian extrinsic geometric analysis}\label{sec:substatic_weighted}
Next,
we briefly review the extrinsic geometric analysis in the literature of Lorentzian geometry.

The \textit{apparent horizon} in general relativity can be formulated in terms of the \textit{marginally outer trapped surface} (MOTS).
Let $(L,\mathfrak{g})$ be an $(n+1)$-dimensional time-oriented Lorentzian manifold,
and let $M$ be a spacelike hypersurface in $L$.
A hypersurface $\Sigma$ in $M$ is called MOTS if the null mean curvature with respect to the outer null normal field vanishes (see e.g., \cite{G2} for the precise definition).
If $M$ is totally geodesic in $L$,
then an MOTS is nothing but a minimal hypersurface in $M$.
Andersson-Mars-Simon \cite{AMS1}, \cite{AMS2} have introduced the notion of the stability for MOTS,
which is generalization of that for minimal hypersurface.
Several geometric properties of stable MOTS have been examined under DEC (see e.g., \cite{G2} and references therein, and also recent works \cite{EGM}, \cite{GM} for \textit{initial data sets}).

For a hypersurface $\Sigma$ in a Riemannian manifold $(M,\mathsf{g})$,
and a function $f\in C^{\infty}(M)$,
the following are equivalent (see e.g., \cite[Proposition 14]{CMZ2}, and also Remark \ref{rem:conformal} below):
\begin{itemize}\setlength{\itemsep}{+1.0mm}
\item\label{item:Main1} $\Sigma$ is stable minimal in $(M,\mathsf{g})$;
\item\label{item:Main2} $\Sigma$ is stable $f$-minimal in a weighted manifold $(M,g,f)$ with $g=\e^{\frac{2f}{n-1}}\mathsf{g}$.
\end{itemize}
Therefore,
the study of stable $f$-minimal hypersurfaces in weighted manifolds of non-negative $1$-weighted Ricci curvature corresponds to that of stable minimal hypersurfaces in substatic triples via conformal change of metric,
and it would be expected to stimulate the progress of the theory of stable MOTS (see also \cite{CFMR}, \cite{F}).

\subsection{Aims}
The present paper is devoted to the study of geometric properties concerning stable $f$-minimal hypersurfaces in weighted manifolds under a lower $1$-weighted Ricci curvature bound.
Such properties have been already examined under a curvature condition
\begin{equation}\label{eq:infty}
\Ric^\infty_f\geq Kg
\end{equation}
for $K\in \mathbb{R}$ (see \cite{BSW}, \cite{CMZ1}, \cite{CMZ2}, \cite{IR1}, \cite{IR2}, \cite{IRS}, \cite{LW}, \cite{Liu2} and references therein).
The study under this setting may be contributed to that of the so-called \textit{self-shrinker},
which is a self-similar solution to the mean curvature flow,
and also can be viewed as an $f$-minimal hypersurface in the Gaussian space.
We will study our object while keeping a curvature condition
\begin{equation}\label{eq:WY}
\Ric^{1}_{f}\geq (n-1)\, \kappa\, \e^{-\frac{4f}{n-1}}\,g
\end{equation}
for $\kappa \in \mathbb{R}$ in mind,
which was introduced by Wylie \cite{W2}, Wylie-Yeroshkin \cite{WY} from the viewpoint of the study of affine connections (see also the works by Kuwae-Li \cite{KL} and Lu-Minguzzi-Ohta \cite{LMO} for the interpolation between the curvature-dimension condition and \eqref{eq:WY}).
They developed comparison geometry under the curvature bound \eqref{eq:WY},
and $f$-minimal hypersurfaces could be found in model spaces (see Remarks \ref{rem:Busemann} and \ref{rem:positive_model} below).
Notice that
the curvature bound \eqref{eq:WY} with $\kappa=0$ is weaker than the classical one \eqref{eq:infty} with $K=0$ by the monotonicity \eqref{eq:monotonicity}.
We extend several previous results under the curvature bound \eqref{eq:infty} to our setting.
As explained in the previous subsection,
our results may be contributed to Lorentzian extrinsic geometric analysis.

The present paper is organized as follows:
In Section \ref{sec:Preliminaries},
we review basics of stability of $f$-minimal hypersurfaces (see Subsection \ref{subsec:stability}).
We also present a key ingredient of the proof concerning the revision of stability inequality in terms of the $1$-weighted Ricci curvature (see Lemma \ref{lem:key}).
We further collect comparison geometric results under the curvature bound \eqref{eq:WY} (see Subsection \ref{subsec:comparison}).
In Section \ref{sec:splitting},
we first study a Schoen-Yau type criterion for stable $f$-minimal hypersurfaces in weighted manifolds of non-negative $1$-weighted Ricci curvature (see Subsection \ref{subsec:SY}).
Having it at hand,
we investigate the structure of three-dimensional weighted manifolds of non-negative $1$-weighted Ricci curvature (see Theorem \ref{thm:Liu}).
In Section \ref{sec:compactness},
we provide non-existence results for stable $f$-minimal hypersurfaces in weighted manifolds under the positivity of $1$-weighted Ricci curvature (see Subsection \ref{sec:volume}).
We apply them to conclude a smooth compactness theorem (see Theorem \ref{thm:LW}).

\section{Preliminaries}\label{sec:Preliminaries}

\subsection{Stability of weighted minimal hypersurfaces}\label{subsec:stability}
We first recall the formulation of weighted minimal hypersurface,
and its stability.
Let $(M,g,f)$ be an $n$-dimensional weighted manifold.
The \textit{weighted Laplacian} is defined by
\begin{equation*}
    \Delta_f := \Delta - \langle \nabla f, \nabla \cdot\rangle,
\end{equation*}
which is associated with the weighted measure defined as \eqref{eq:volume}.

Let $\Sigma$ be an immersed hypersurface in $M$.
The second fundamental form $\second_{\Sigma}$,
the mean curvature vector $\mathbf{H}_{\Sigma}$ and the weighted mean curvature vector $\mathbf{H}_{f,\Sigma}$ are defined as 
\begin{equation*}
\second_{\Sigma}(X,Y) := (\nabla_XY)^{\perp}, \quad \mathbf{H}_{\Sigma} := \tr\, \second_{\Sigma},\quad \mathbf{H}_{f,\Sigma} := \mathbf{H}_{\Sigma} +(\nabla f)^{\perp}.
\end{equation*}
An immersed hypersurface $\Sigma$ is said to be \textit{$f$-minimal} if it is a critical point of the weighted volume functional.
By the first variation formula (see e.g., \cite[Proposition 1]{Liu2}, \cite[Proposition 2]{CMZ2}),
$\Sigma$ is $f$-minimal if and only if $\mathbf{H}_{f,\Sigma} \equiv 0$.
An immersed $f$-minimal hypersurface $\Sigma$ is called \textit{stable} if for every compactly supported normal variation of $\Sigma$,
the second derivative of the weighted volume functional at $\Sigma$ is non-negative.
Due to the second variation formula (see e.g., \cite[Proposition 1]{Liu2}, \cite[Proposition 13]{CMZ2}),
$\Sigma$ is stable if and only if the index form
\begin{equation*}
Q_{f,\Sigma}(T,T):=\int_{\Sigma}\left\{ |\nabla^{\perp}T|^2-\left(   \Ric^\infty_f(T,T)+|\second_{\Sigma}|^2  |T|^2\right)   \right\}\,\d m_{f,\Sigma}
\end{equation*}
is non-negative for every compactly supported normal vector field $T$,
where $\nabla^{\perp}$ is the connection on normal bundle.
For an immersed $f$-minimal hypersurface $\Sigma$,
the \textit{$f$-index} denoted by $\Ind_f(\Sigma)$ is defined to be the maximal dimension of the space on which the index form is negative definite.

Let $\Sigma$ be a two-sided immersed hypersurface in $M$.
We denote by $\nu$ a globally defined unit normal vector field on $\Sigma$.
The (scalar) second fundamental form $\second_{\Sigma}$,
the mean curvature $H_{\Sigma}$ and the weighted mean curvature $H_{f,\Sigma}$ are defined as 
\begin{equation*}
 \second_{\Sigma}(X,Y) := \langle \nabla_X \nu,Y \rangle, \quad H_{\Sigma} := \tr\, \second_{\Sigma},\quad H_{f,\Sigma} := H_{\Sigma} - f_\nu.
\end{equation*}
Notice that $\mathbf{H}_{f,\Sigma}=-H_{f,\Sigma}\,\nu$.
A two-sided immersed hypersurface $\Sigma$ is $f$-minimal if and only if $H_{f,\Sigma} \equiv 0$.
Further,
a two-sided immersed $f$-minimal hypersurface $\Sigma$ is stable if and only if 
\begin{equation}\label{eq:stable}
Q_{f,\Sigma}(\phi,\phi) := \int_\Sigma \phi \mathcal{L}_{f,\Sigma}\phi\ \d m_{f,\Sigma} \geq 0
\end{equation}
for every $\phi \in C^{\infty}_0(\Sigma)$,
where $\mathcal{L}_{f,\Sigma}$ is the \textit{stability operator} defined by 
\begin{equation*}
\mathcal{L}_{f,\Sigma}\, \phi := - \Delta_{f,\Sigma} \,\phi - \left( \Ric_f^\infty(\nu,\nu)+ \left|\second_\Sigma\right|^2 \right)\phi.
\end{equation*}
Note that
the stability inequality \eqref{eq:stable} can be expressed by
\begin{equation*}\label{eq:stable}
\int_\Sigma \left( \Ric_f^\infty(\nu,\nu)+ \left|\second_\Sigma\right|^2 \right)\phi^2\,\d m_{f,\Sigma}\leq \int_\Sigma \,|\nabla_\Sigma \phi|^2 \,\d m_{f,\Sigma}.
\end{equation*}

We now present a key ingredient of our proof concerning the revision of stability operator.
\begin{lem}\label{lem:key}
Let $\Sigma$ be a two-sided immersed $f$-minimal hypersurface.
Then we have
\begin{equation*}
\Ric_f^\infty(\nu,\nu) + \left| \second_\Sigma \right|^2 = \Ric_f^1(\nu,\nu) + \left| \second_\Sigma - \frac{f_\nu}{n-1}g_\Sigma \right|^2.
\end{equation*}
In particular, the stability inequality can be expressed as
\begin{equation}\label{eq:key}
\int_\Sigma \left( \Ric_f^1(\nu,\nu) + \left| \second_\Sigma - \frac{f_\nu}{n-1}g_\Sigma \right|^2 \right)\phi^2\,\d m_{f,\Sigma}\leq \int_\Sigma \,|\nabla_\Sigma \phi|^2 \,\d m_{f,\Sigma}.
\end{equation}
\end{lem}
\begin{proof}
Since $\Sigma$ is $f$-minimal, 
we possess $H_{\Sigma} = f_\nu$,
and hence
\begin{equation*}
 \left| \second_\Sigma - \frac{f_\nu}{n-1}g_{\Sigma} \right|^2= |\second_\Sigma|^2 - \frac{2 f_\nu H_\Sigma}{n-1} + \frac{f_\nu^2}{n-1}= |\second_\Sigma|^2 - \frac{f_\nu^2}{n-1}.
\end{equation*}
It follows that
\begin{align*}
    \Ric_f^{\infty}(\nu,\nu) + |\second_\Sigma|^2
    &= \Ric_f^\infty(\nu,\nu) + \frac{f_\nu^2}{n-1} + \left| \second_\Sigma  -\frac{f_\nu}{n-1} g_\Sigma\right|^2\\
    &= \Ric_f^1 (\nu,\nu)+ \left| \second_\Sigma  -\frac{f_\nu}{n-1}g_{\Sigma} \right|^2.
\end{align*}
This completes the proof.
\end{proof}

\begin{rem}
This observation also can be found in recent works by Colombo-Freitas-Mari-Rigoli \cite{CFMR} and the first named author \cite{F} (see \cite[Section 3]{CFMR} and \cite[Proposition 2.3]{F}).
\end{rem}

Let $\Sigma$ be a complete two-sided immersed $f$-minimal hypersurface.
For a relatively compact domain $\Omega$ in $\Sigma$,
we denote by $\Ind_f(\Omega)$ the number of negative eigenvalues of the operator $\mathcal{L}_{f,\Sigma}$ on $C_0^\infty(\Omega)$.
The $f$-index of $\Sigma$ is given by
\begin{equation*}
    \Ind_f(\Sigma)= \sup_{\Omega \subset \Sigma} \Ind_f(\Omega),
\end{equation*}
where the supremum is taken over all relatively compact domains $\Omega$ in $\Sigma$.

\begin{rem}\label{rem:conformal}
We here give a remark on the behavior of the above notions under the conformal change of metric.
Let $(M,\mathsf{g})$ be a Riemannian manifold,
and let $f\in C^{\infty}(M)$.
We consider a conformally changed metric $g=\e^{\frac{2f}{n-1}}\mathsf{g}$,
and an associated weighted manifold $(M,g,f)$.
Let $\Sigma$ stand for a two-sided immersed hypersurface in $M$.
By direct calculations,
we see
\begin{equation*}
\second_{\g,\Sigma} = \e^{-\frac{f}{n-1}} \left( \second_\Sigma- \frac{f_\nu}{n-1}g_{\Sigma} \right),\quad H_{\g,\Sigma} = \e^{\frac{f}{n-1}} H_{f,\Sigma}.
\end{equation*}
In particular,
$\Sigma$ is totally geodesic in $(M,\g)$ if and only if it satisfies
\begin{equation}\label{eq:second}
\second_\Sigma \equiv  \frac{f_\nu}{n-1}g_{\Sigma}
\end{equation}
in $(M,g,f)$.
Also,
$\Sigma$ is minimal in $(M,\g)$ if and only if it is $f$-minimal in $(M,g,f)$.
Moreover,
in this case,
we possess the following (see e.g., \cite[Appendix]{CMZ2}):
\begin{equation*}
    V\,\mathcal{L}_{\g,\Sigma}\left( V\phi \right) =\mathcal{L}_{f,\Sigma}(\phi),\quad Q_{\g,\Sigma}\left(V\phi,V\phi\right) = Q_{f,\Sigma}(\phi,\phi),\quad \Ind_{\g}(\Sigma) = \Ind_f(\Sigma);
\end{equation*}
in particular,
$\Sigma$ is stable minimal hypersurface in $(M,\g)$ if and only if it is stable $f$-minimal hypersurface in $(M,g,f)$.

Notice that
if we consider a substatic triple $(M,\mathsf{g},V)$ defined as \eqref{eq:substatic2},
then we will choose $f=-(n-1)\log V$,
and in this case,
we possess the following (see \cite[Appendix A.2]{BF}):
\begin{equation}\label{eq:BF_relation}
\Ric^1_f=\Ric_{\mathsf{g}}-\frac{\Hess_{\mathsf{g}} V}{V}+\frac{\Delta_\mathsf{g} V}{V} \mathsf{g}.
\end{equation}
In particular,
one can verify the equivalence stated in Subsection \ref{subsec:Lorentzian}.
More generally,
for $\kappa \in \mathbb{R}$,
the following are equivalent:
\begin{itemize}\setlength{\itemsep}{+1.0mm}
\item $(M,\mathsf{g},V)$ satisfies a curvature condition
\begin{equation}\label{eq:substatic_WY}
V\Ric_{\mathsf{g}} - \Hess_{\mathsf{g}} V + (\Delta_{\mathsf{g}} V)\mathsf{g}\geq (n-1)\kappa V^3\,\mathsf{g};
\end{equation}
\item $(M,g,f)$ satisfies the curvature condition \eqref{eq:WY}.
\end{itemize}
The authors do not know whether the curvature condition \eqref{eq:substatic_WY} for $\kappa \neq 0$ is reasonable in the context of the substatic geometry.
One can find a curvature condition
\begin{equation*}\label{eq:Miao-Tam}
V\Ric_{\g} - \Hess_{\g} V + (\Delta_{\g} V)\g= K\,\g
\end{equation*}
for $K\in \mathbb{R}$ in the literature of finding stationary points for the volume functional on the space of constant scalar curvature metrics (see \cite{MT1}, \cite{MT2}, \cite{CEM}).
From this point of view,
the natural curvature condition might be the following (see also \cite[Corollary 1.5]{LX}):
\begin{equation}\label{eq:Miao-Tam}
V\Ric_{\g} - \Hess_{\g} V + (\Delta_{\g} V)\g\geq K\,\g.
\end{equation}
Borghini-Fogagnolo \cite{BF} have also mentioned the validity of \eqref{eq:Miao-Tam} in view of the so-called \textit{Besse conjecture}.
Notice that the curvature condition \eqref{eq:Miao-Tam} on $(M,\g,V)$ is equivalent to
\begin{equation*}
\Ric^1_f\geq K\, \e^{-\frac{f}{n-1}}\,g
\end{equation*}
on $(M,g,f)$ via the relation \eqref{eq:BF_relation}.
\end{rem}

\subsection{Comparison geometry}\label{subsec:comparison}
We next recall the comparison geometry under the curvature bound \eqref{eq:WY} established by Wylie \cite{W2} and Wylie-Yeroshkin \cite{WY} (see also \cite{KS}).

Wylie \cite{W2} has established the following splitting theorem of Cheeger-Gromoll type under the non-negativity of $1$-weighted Ricci curvature (see \cite[Theorem 1.2]{W2}):
\begin{thm}[\cite{W2}]\label{thm:splitting-wylie}
Let $(M,g,f)$ be a complete weighted manifold.
We assume that $\Ric_f^1 \geq 0$ and $f$ is bounded from above.
If $M$ contains a line,
then $(M,g)$ is isometric to a warped product over $\mathbb{R}$.
\end{thm}

\begin{rem}\label{rem:Busemann}
In the splitting case in Theorem \ref{thm:splitting-wylie},
one can observe that
each slice (i.e., each level set of the Busemann function) must be an $f$-minimal hypersurface since the Busemann function is $f$-harmonic and its norm of the gradient is identically one from its proof (see also the example in \cite[Corollary 2.4]{W2}).
\end{rem}

Wylie-Yeroshkin \cite{WY} have developed Laplacian and volume comparison under the curvature bound \eqref{eq:WY} (see \cite[Section 4]{WY}).
They have derived the following:
\begin{prop}[\cite{WY}]\label{prop:volume-comparison}
Let $(M,g,f)$ be an $n$-dimensional complete weighted manifold.
We assume that $\Ric_f^1 \geq 0$ and $f$ is bounded from below.
Let $o\in M$.
Then for all $r>0$ we have
\begin{equation*}
m_{f,\partial B_r(o)}(\partial B_r(o))\leq \omega_{n-1}\,\e^{f(o)-2\inf f}\,r^{n-1},
\end{equation*}
where $\omega_{n-1}$ is the volume of the $(n-1)$-dimensional unit sphere $\mathbb{S}^{n-1}$.
\end{prop}

We next discuss results under the positivity of $1$-weighted Ricci curvature.
Let $(M,g,f)$ be an $n$-dimensional complete weighted manifold.
Wylie-Yeroshkin \cite{WY} have introduced the \textit{reparametrized distance} $d_{f}:M\times M \to \mathbb{R}$ defined by
\begin{equation*}
d_{f}(x,y):=\inf_{\gamma} \int^{d(x,y)}_{0}\,\e^{-\frac{2f(\gamma(\xi))}{n-1}}\,\d\xi,
\end{equation*}
where the infimum is taken over all unit speed minimal geodesics $\gamma:[0,d(x,y)]\to M$ from $x$ to $y$.
They obtained the following estimate of Bonnet-Myers type (see \cite[Theorem 2.2]{WY}):
\begin{prop}[\cite{WY}]\label{prop:BM}
Let $(M,g,f)$ be an $n$-dimensional complete weighted manifold. 
For $\kappa > 0$,
we assume $\Ric_f^1 \geq (n-1)\,\kappa \,\e^{-\frac{4f}{n-1}}g$.
Then we have
\begin{equation*}
\sup_{x,y\in M} d_f(x,y)\leq \frac{\pi}{\sqrt{\kappa}}.
\end{equation*}
In particular,
if $f$ is bounded from above,
then $M$ must be compact.
\end{prop}

\begin{rem}\label{rem:closed}
In Subsection \ref{subsec:compactness} below,
we will focus on a closed weighted manifold $(M,g,f)$ of positive $1$-weighted Ricci curvature.
Then it satisfies the curvature bound \eqref{eq:WY} for some $\kappa>0$,
and the curvature bound can be lifted over the universal cover.
Since the lift of $f$ is also bounded from above,
Proposition \ref{prop:BM} says that the universal cover must be compact;
in particular,
the fundamental group $\pi_1(M)$ must be finite.
Notice that
the finiteness of the fundamental group can be available under a weaker setting that $(M,g,f)$ is complete and satisfies the curvature bound \eqref{eq:WY} for $\kappa>0$ (see \cite[Theorem 2.9]{WY}).
\end{rem}

\begin{rem}\label{rem:positive_model}
Under the same setting as in Proposition \ref{prop:BM},
Wylie-Yeroshkin \cite{WY} have produced a rigidity theorem of Cheng type concerning the equality case by setting a model space as a warped product of the following form (see \cite[Theorem 2.6]{WY}):
\begin{equation*}
\left([0,D],dt^2+\exp\left( 2\frac{f(t)+f(0)}{n-1} \right)\,\frac{\sin^2\left(\sqrt{\kappa}\,s_{f}(t)\right)}{\kappa}g_{\mathbb{S}^{n-1}}\right)
\end{equation*}
with
\begin{equation*}
s_f(t):=\int^t_0\,\e^{-\frac{2f(\xi)}{n-1}}\,\d\xi,\quad s_f(D)=\frac{\pi}{\sqrt{\kappa}}.
\end{equation*}
On this space,
the slice at $t_0$ determined by $s_f(t_0)=\pi/(2\sqrt{\kappa})$ is $f$-minimal.
\end{rem}

We close this subsection with the following Frankel property (see \cite[Proposition 3.6]{FS}):
\begin{prop}[\cite{FS}]\label{prop:Frankel}
Let $(M,g,f)$ be a complete weighted manifold of positive $1$-weighted Ricci curvature.
Let $\Sigma_1$ and $\Sigma_2$ be two complete properly immersed $f$-minimal hypersurfaces in $M$ such that at least one of them is compact.
Then $\Sigma_1$ and $\Sigma_2$ must intersect.
\end{prop}

\section{Schoen-Yau type criterion and its application to three-manifolds}\label{sec:splitting}

\subsection{Schoen-Yau type criterion}\label{subsec:SY}

In the present subsection,
we present a Schoen-Yau type criterion for stable $f$-minimal hypersurfaces in weighted manifolds of non-negative $1$-weighted Ricci curvature.
Such a result has been obtained by Liu \cite{Liu2} under the non-negativity of $\infty$-weighted Ricci curvature (see \cite[Sections 3, 5]{Liu2}).
We generalize it under a weak assumption.

We first show the following (cf. \cite[Theorem 1 and Proposition 2]{Liu2}):
\begin{lem}\label{lem:SY_quadratic}
Let $(M,g,f)$ be an $n$-dimensional complete orientable weighted manifold.
Let $\Sigma$ be a complete orientable immersed stable $f$-minimal hypersurface in $M$.
We assume that 
$\Ric_f^1 \geq 0$ along $\Sigma$.
Let $o\in \Sigma$.
We also assume that
for every $r > 0$ we have
\begin{equation*}
m_{f,\Sigma}(B_r^{\Sigma}(o)) \leq C r^2
\end{equation*}
for some $C > 0$,
where $B^\Sigma_r(o)$ is the intrinsic open geodesic ball in $\Sigma$ of radius $r$ centered at $o$.
Then we have 
\begin{equation}\label{eq:SY_second}
 \second_\Sigma \equiv \frac{f_\nu}{n-1}g_{\Sigma}, \quad \Ric_f^1 (\nu,\nu) \equiv 0.
\end{equation}
\end{lem}
\begin{proof}
We adopt the logarithmic cutoff trick.
For $\SL >0$ we define
\begin{align*}
    \eta_{\SL}(t) :=
    \begin{cases}
        1 & \mbox{ for } t\in [0,\e^{\SL}],\\
        2-\displaystyle{\frac{\log t}{\SL}} &\mbox{ for } t\in (\e^{\SL},\e^{2\SL}),\\
        0 & \mbox{ for } t\in [\e^{2\SL},+\infty),
    \end{cases}
    \qquad \phi_{\SL} := \eta_{\SL} \circ d_o^\Sigma,
\end{align*}
where $d_o^{\Sigma}$ denotes the intrinsic distance function over $\Sigma$ from $o$.
Notice that
\begin{equation*}
\sup_{B^{\Sigma}_{\e^{\ell}}(o)\setminus B^{\Sigma}_{\e^{\ell-1}}(o)}|\nabla_{\Sigma} \phi_{\SL} |^2\leq {\SL}^{-2}\,\e^{2-2\ell}
\end{equation*}
for $\ell=\SL+1,\dots,2\SL$.
From the stability inequality \eqref{eq:key},
it follows that
\begin{align*}
 0 &\leq \int_\Sigma \left( \Ric_f^1(\nu,\nu) + \left| \second_\Sigma - \frac{f_\nu}{n-1} g_{\Sigma}\right|^2 \right)\phi_{\SL}^2 \ \d m_{f,\Sigma}\\
 &\leq \int_{\Sigma} |\nabla_{\Sigma}\, \phi_{\SL}|^2 \ \d m_{f,\Sigma} \leq \sum^{2\SL}_{\ell=L+1}\left(\SL^{-2}\,\e^{2-2\ell}\,m_{f,\Sigma} (B^{\Sigma}_{\e^{\ell}}(o)) \right)\leq \frac{C\e^2}{\SL}.
\end{align*}
By letting $\SL \rightarrow + \infty$,
we arrive at the desired conclusion.
\end{proof}

\begin{rem}
By \eqref{eq:second} in Remark \ref{rem:conformal},
the conclusion for second fundamental form in \eqref{eq:SY_second} is equivalent to the condition that $\Sigma$ is totally geodesic with respect to the conformally changed metric $\mathsf{g}:=\e^{-\frac{2f}{n-1}}g$.
\end{rem}

\begin{rem}
In \cite[Theorem 1]{Liu2},
he has also shown a rigidity result for weighed manifold of non-negative $\infty$-weighted Ricci curvature under the existence of a certain compact $f$-area-minimizing hypersurface.
Similarly,
we have the following:
Let $(M,g,f)$ be an $n$-dimensional complete orientable weighted manifold.
We assume $\Ric^1_f\geq 0$.
Let $\Sigma$ be a closed orientable embedded stable $f$-minimal hypersurface in $M$.
If $\Sigma$ is $f$-area-minimizing in its homology class,
then $M$ is isometric to a quotient of warped product in the form of 
\begin{equation}\label{eq:compact}
\left(\mathbb{R}\times \Sigma,dt^2+\e^{2\frac{f(\gamma_z(t))-f(z)}{n-1}}g_{\Sigma}  \right)
\end{equation}
with $f(\gamma_z(t))=f_1(t)+f_2(z)$ for some $f_1:\mathbb{R}\to \mathbb{R}$ and $f_2:\Sigma\to \mathbb{R}$,
where $\gamma_z$ is the line with $\gamma_z(0)=z$ that is orthogonal to $\Sigma$.
The outline of the proof is as follows:
Let $\{\Sigma_t\}_{t\in (-\eps,\eps)}$ be a family of hypersurfaces in $M$ such that $\Sigma_t$ is the $t$-level set of the signed distance function from $\Sigma$ in $M$.
On the side such that $t\geq 0$,
the weighted Laplacian of the distance function from $\Sigma$ is non-positive due to the Laplacian comparison (see e.g., \cite[Theorem 5.2]{W2}, \cite[Lemmas 3.3 and 3.8]{Sa1});
in particular,
the weighted measure of $\Sigma_t$ is non-increasing in $t$.
Since $\Sigma=\Sigma_0$ is $f$-area-minimizing,
it must be constant in $t$,
and hence the equality in the Laplacian comparison holds,
and the rigidity for the equality case can be used.
By applying the same argument to the opposite side,
and by induction,
we arrive at the desired assertion thanks to \cite[Proposition 2.2]{W2}.
\end{rem}

One can conclude the following (cf. \cite[Proposition 3]{Liu2}):
\begin{prop}\label{prop:SY}
Let $(M,g,f)$ be a three-dimensional complete orientable weighted manifold.
We assume that $\Ric_f^1 \geq 0$ and $f$ is bounded from below.
Let $\Sigma$ be a complete orientable immersed $f$-area-minimizing surface in $M$.
Then we have 
\begin{equation}\label{eq:SY}
\second_\Sigma \equiv \frac{f_\nu}{2}g_{\Sigma},\quad \Ric_f^1 (\nu,\nu) \equiv 0.
\end{equation}
\end{prop}
\begin{proof}
Let $o\in \Sigma$ and $r >0$.
Since $\Sigma$ is $f$-area-minimizing,
Proposition \ref{prop:volume-comparison} implies
\begin{equation}\label{eq:quadratic}
    m_{f,\Sigma}(B_r^\Sigma(o)) \leq m_{f,\Sigma}(B_r(o)\cap \Sigma) \leq m_{f,\partial B_r(o)}(\partial B_r(o)) \leq \omega_{2}\,\e^{f(o)-2\inf f}\,r^{2}.
\end{equation}
Therefore,
Lemma \ref{lem:SY_quadratic} leads us to the desired assertion.
\end{proof}

\subsection{Structure of weighted three-manifolds}\label{subsec:3D}
The structure of complete three-manifolds of non-negative Ricci curvature has been investigated by Hamilton \cite{H1}, \cite{H2}, Schoen-Yau \cite{SY2} and Liu \cite{Liu1}.
As a consequence,
Liu \cite{Liu1} has concluded that
the Milnor conjecture is true in the three dimensional case (see \cite[Corollary 1]{Liu1}).
In \cite{Liu2},
he has also examined the structure of three-dimensional weighted manifolds of non-negative $\infty$-weighted Ricci curvature via the analysis of stable $f$-minimal hypersurfaces (see \cite[Theorem 3]{Liu2}).
We generalize his result for weighted manifolds of non-negative $1$-weighted Ricci curvature.
\begin{thm}\label{thm:Liu}
Let $(M,g,f)$ be a three-dimensional complete weighted manifold.
We assume that $\Ric_f^1 \geq 0$ and $f$ is bounded.
Then the following holds:
\begin{enumerate}\setlength{\itemsep}{+3mm}
\smallskip

\item[$(1)$] If $M$ is compact, then either
  \begin{enumerate}\setlength{\itemsep}{+3mm}
  \medskip
  
  \item[$(1a)$] the universal cover is diffeomorphic to $\mathbb{S}^3$; or
  \item[$(1b)$] the universal cover is isometric to a warped product over $\mathbb{R}$;
  \end{enumerate}
 \item[$(2)$] if $M$ is non-compact, then either
   \begin{enumerate}\setlength{\itemsep}{+3mm}
   \medskip
   
  \item[$(2a)$] the universal cover is isometric to a warped product over $\mathbb{R}$; or
  \item[$(2b)$] $M$ is contractible; or
  \item[$(2c)$] for each $o\in M$,
          there exists a surface $\Sigma$ passing through $o$ such that 
          \begin{equation}\label{eq:surface}
            \second_{\Sigma} \equiv \frac{f_\nu}{2}g_\Sigma, \quad \Ric_f^1(\nu,\nu) \equiv 0.
          \end{equation}
          \end{enumerate}
\end{enumerate}
\end{thm}
\begin{proof}
We give a proof along the line of the argument in \cite{Liu1}, \cite{Liu2}.

We first consider the case where $M$ is compact.
If the universal cover of $M$ is compact,
then it is diffeomorphic to $\mathbb{S}^3$ due to the resolution of the Poincar\'{e} conjecture (\cite{P1}, \cite{P2}, \cite{P3}).
If the universal cover is non-compact,
then it contains a line (see e.g., \cite[Theorem 4.5]{W2}),
and hence Theorem \ref{thm:splitting-wylie} tells us that the universal cover is isometric to a warped product.

We now study the case where $M$ is non-compact.
If $\pi_2(M)$ is non-trivial,
then the universal cover contains a line (see \cite[Lemma 2]{SY2}),
and thus the universal cover is isometric to a warped product by Theorem \ref{thm:splitting-wylie}.
If $\pi_2(M)$ is trivial,
then by the Hurewicz theorem,
$M$ is a $K(\pi,1)$ space since $M$ is three dimensional and non-compact (see \cite[Lemma 3]{SY2});
in particular,
if $\pi_1(M)$ is also trivial,
then $M$ is contractible thanks to the Whitehead theorem.
Therefore,
it suffices to discuss the case where $\pi_1(M)$ is non-trivial.

We assume that $\pi_2(M)$ is trivial, and $\pi_1(M)$ is non-trivial.
From the the fact that $M$ is a $K(\pi,1)$ space,
it follows that
$\pi_1(M)$ is torsion free;
in particular,
we may assume that $\pi_1(M) = \mathbb{Z}$ and $M$ is orientable by passing to a suitable covering.
Let $\gamma$ be a Jordan curve that represents the generator of $\pi_1(M)$.
By the argument of Anderson \cite{A},
there is a complete orientable immersed $f$-area-minimizing surface in $M$ that intersects with $\gamma$ (see \cite[Lemma 2.2]{A}, and also \cite[Lemma 2]{Liu2}).
From Proposition \ref{prop:SY},
it follows that \eqref{eq:SY} holds along it.
If $\Ric^1_f>0$ on $M$,
then this is a contradiction.

We consider the case of $\Ric^1_f\geq 0$.
Let $o\in M$.
We may assume that
the point $o$ does not lie on $\gamma$ by perturbing $\gamma$ if necessary.
We adopt the metric deformation argument of Ehrlich \cite{E}.
Liu \cite{Liu2} has extended the argument to the weighted case for $\Ric^{\infty}_f$.
Similarly to \cite{E}, \cite{Liu2},
for $R > 0$,
we set $\eta(r) := R - r$ for $r \in [R/2,R]$,
and extend it smoothly to $[0,R/2)$ such that it is positive.
Furthermore,
for $t \geq 0$ we define
\begin{equation*}
\rho := \eta\circ d_o,\quad g_t := \exp(- 2 t \rho^5)g
\end{equation*}
for the distance function $d_o$ on $M$ from $o$,
where we extend $g_t$ such that it does not change outside $B_R(o)$.
For every $x \in B_R(o)\backslash B_{7R/8}(o)$,
and for every unit vector $v$ at $x$ with respect to $g$,
the argument implies
\begin{equation*}
 \Ric_{f,t}^\infty(v,v) \geq \e^{2t\rho^5}\left\{\Ric_{f}^{\infty}(v,v) + 20 t \rho^3+5 t \rho^4\left(\Delta \rho + \Hess\rho(v, v)\right)-25 t^2 \rho^8 \right\}-15 t \rho^4|\nabla f|,
\end{equation*}
where $\Ric^{\infty}_{f,t}$ denotes the $\infty$-weighted Ricci curvature for $(M,g_t,f)$.
It follows that
\begin{align*}
 \Ric_{f,t}^1(v,v) &= \Ric_{f,t}^{\infty}(v,v) + \frac{\d f(v)^2}{2}\\
 &\geq  \e^{2t\rho^5}\left\{\Ric^1_{f}(v,v) + 20 t \rho^3+5 t \rho^4\left(\Delta \rho + \Hess\rho(v, v)\right)-25 t^2 \rho^8 \right\}\\
 &\qquad \qquad\qquad \qquad\qquad \qquad \qquad -15 t \rho^4|\nabla f|-\frac{\d f(v)^2}{2}(\e^{2t\rho^5}-1)\\
 &\geq  \e^{2t\rho^5}\left\{\Ric^1_{f}(v,v) + 20 t \rho^3+5 t \rho^4\left(\Delta \rho + \Hess\rho(v, v)\right)-25 t^2 \rho^8 \right\}\\
 &\qquad \qquad\qquad \qquad\qquad \qquad \qquad -t\rho^4\left( 15|\nabla f|+(\e-1)|\nabla f|^2 \rho  \right)
\end{align*}
if $t\leq R^{-5}/2$.
Hence,
similarly to \cite{Liu2},
one can conclude that
there exists a sufficiently small $R>0$ such that
for every sufficiently small $t$,
the metric $g_t$ has positive $1$-weighted Ricci curvature over $B_R(o)\backslash B_{7R/8}(o)$.
We apply this perturbation finitely many times such that
the $1$-weighted Ricci curvature is positive along $\gamma$,
and also non-negative except for a small neighborhood of $o$.
For this new metric,
we construct a complete orientable immersed $f$-area-minimizing surface in $M$ that intersects with $\gamma$ by using the argument in \cite{A} again.
Notice that
it also has the quadratic $f$-volume growth by Proposition \ref{prop:volume-comparison} since the new metric is uniformly equivalent to the original one $g$ (cf. \cite[(6.3)]{Liu2}).
It follows that
the surface must pass through the small neighborhood of $o$.
Indeed,
if this is not true,
then the $1$-weighted Ricci curvature is non-negative along the surface,
and also positive at the intersection with $\gamma$;
however,
this contradicts with Lemma \ref{lem:SY_quadratic}.

Let us shrink the size of the neighborhood of $o$ such that $1$-weighted Ricci curvature might be negative.
We obtain a sequence of metrics,
and for each metric,
we also obtain a complete orientable immersed $f$-area-minimizing surface as above.
We may assume that
this sequence of metrics converges to the original one $g$ in $C^4$-sense.
We may also assume that
the sequence of surfaces converges to a complete orientable immersed $f$-area-minimizing surface $\Sigma$ in $M$ with respect to $g$ passing through $o$.
Thanks to Proposition \ref{prop:SY},
$\Sigma$ satisfies \eqref{eq:SY}.
Thus,
we complete the proof.
\end{proof}

\begin{rem}
In the case of $(1b)$ and $(2a)$,
the cross section of the splitting is conformal to either $\mathbb{S}^2$ or $\mathbb{C}$.
This can be proved by the same argument as in \cite{Liu2} since it has the quadratic volume growth in view of \eqref{eq:quadratic} and the upper boundedness of $f$.
\end{rem}

\begin{rem}
In the sequel to the proof of $(2b)$ and $(2c)$ in Theorem \ref{thm:Liu},
if we further assume that
the rank of $\Ric^1_f$ is at least $2$ everywhere,
then they can be improved as follows:
\begin{enumerate}\setlength{\itemsep}{+2mm}
\item[$(2b)^{\ast}$] $M$ is diffeomorphic to $\mathbb{R}^3$;
 \item[$(2c)^{\ast}$] the universal cover is isometric to a warped product over $\mathbb{R}$ (same as (2a)).
 \end{enumerate}

We first discuss $(2c)^{\ast}$.
By the additional assumption,
we see that
for each $o\in M$,
the surface $\Sigma$ is uniquely determined,
and $M$ is foliated by them (cf. \cite[Section 3]{Shi}).
Let $\{\Sigma_t\}_t$ stand for the foliation.
Notice the weighted evolution formula
\begin{align*}
\partial_t H_{f,\Sigma_t}&=\partial_t H_{\Sigma_t}- \partial_t f_\nu\\
&=\left\{-\Delta_{\Sigma_t} \,\phi - \left( \Ric(\nu,\nu)+ \left|\second_{\Sigma_t}\right|^2 \right)\phi \right\}-\left\{\phi \nabla^2 f(\nu,\nu)-\langle \nabla_{\Sigma_t}f,\nabla_{\Sigma_t}\phi  \rangle   \right\}\\
&=-\Delta_{f,\Sigma_t} \,\phi - \left( \Ric_f^\infty(\nu,\nu)+ \left|\second_{\Sigma_t}\right|^2 \right)\phi
\end{align*}
holds for the variation field $\phi \,\nu$ on $\Sigma_t$,
which can be derived from the evolution formula for the mean curvature and unit normal vector filed in the unweighted case (see e.g., \cite[Theorem 3.2]{HuP}). 
From the $f$-minimality, \eqref{eq:surface} and Lemma \ref{lem:key},
it follows that
\begin{equation*}
0=\Delta_{f,\Sigma_t} \,\phi + \left( \Ric_f^\infty(\nu,\nu)+ \left|\second_{\Sigma_t}\right|^2 \right)\phi
=\Delta_{f,\Sigma_t} \,\phi + \left( \Ric_f^1(\nu,\nu)+\left| \second_{\Sigma_t} - \frac{f_\nu}{2}g_{\Sigma_t} \right|^2 \right)\phi=\Delta_{f,\Sigma_t}\,\phi.
\end{equation*}
By a Liouville result for $f$-harmonic functions together with \eqref{eq:quadratic},
$\phi$ must be constant (see \cite[Lemma 3]{Liu2}).
By the re-parametrization,
we may assume $\phi \equiv 1$,
and conclude $\nabla_\nu \nu=0$.
Notice that
the second fundamental form of $\Sigma_t$ coincides with the Hessian of the distance function from $\Sigma_0$,
and satisfies \eqref{eq:surface}.
Thus,
by the same argument as in the proof of Theorem \ref{thm:splitting-wylie} in \cite{W2},
the universal cover is isometric to a warped product over $\mathbb{R}$ in the form of
\begin{equation*}
\left(\mathbb{R}\times \Sigma_0,dt^2+\e^{f(\gamma_z(t))-f(z)}g_{\Sigma_0}  \right)
\end{equation*}
with $f(\gamma_z(t))=f_1(t)+f_2(z)$ for some $f_1:\mathbb{R}\to \mathbb{R}$ and $f_2:\Sigma_0\to \mathbb{R}$,
where $\gamma_z$ is the line with $\gamma_z(0)=z$ that is orthogonal to $\Sigma_0$.
This proves the desired assertion.

We next observe $(2b)^{\ast}$.
In the sequel to the proof of $(2b)$,
$M$ must be irreducible with the help of the resolution of the Poincar\'{e} conjecture (see e.g., \cite[Lemma 2.2]{CLX}).
Hence,
by a topological result by Husch-Price \cite{HusP},
it suffices to prove that $M$ is simply connected at infinity.
Suppose that
$M$ is not simply connected at infinity.
Fix $o\in M$,
and a compact subset $\mathcal{K}$ of $M$ to which $o$ does not belong.
We apply the above metric deformation argument to $\mathcal{K}$ instead of $\gamma$,
and obtain a metric such that the $1$-weighted Ricci curvature is positive over $\mathcal{K}$,
and also non-negative except for a small neighborhood of $o$.
Since $M$ is not simply connected at infinity,
by the same argument as in \cite{Liu1},
there exists a complete orientable immersed $f$-area-minimizing surface in $M$ that intersects with $\mathcal{K}$ and passes through the small neighborhood of $o$.
By shrinking the size of the neighborhood of $o$ such that the $1$-weighted Ricci curvature might be negative,
we obtain a complete orientable immersed $f$-area-minimizing surface $\Sigma$ in $M$ with respect to $g$ passing through $o$ and satisfying \eqref{eq:surface}.
If the rank of $\Ric^1_f$ is at least $2$ everywhere,
then one can apply the argument in the above paragraph,
and conclude that $M$ is isometric to a warped product over $\mathbb{R}$.
This contradicts with the assumption that $M$ is not simply connectedness at infinity.
\end{rem}

\begin{rem}
The assumption for upper boundedness of $f$ in Theorem \ref{thm:splitting-wylie} has been weakened as the $f$-completeness introduced in \cite{W2}, \cite{WY} (cf. \cite[Remark 1.2]{W2}). 
In the same manner,
the assumption in Theorem \ref{thm:Liu} also can be weakened as the combination of the $f$-completeness and the lower boundedness of $f$.
\end{rem}

\section{Non-existence and smooth compactness}\label{sec:compactness}

In this section,
we show non-existence results for stable $f$-minimal hypersurfaces in weighted manifolds under the positivity of $1$-weighted Ricci curvature.
Furthermore,
we prove smooth compactness theorems.

\subsection{Non-existence}\label{sec:volume}

We begin with the following (cf. \cite[Theorem 3]{CMZ2}):
\begin{prop}\label{prop:non-existence}
Let $(M,g,f)$ be a complete weighted manifold of positive $1$-weighted Ricci curvature.
Then there is no complete two-sided immersed stable $f$-minimal hypersurface $\Sigma$ in $M$ with $m_{f,\Sigma}(\Sigma) < + \infty$.
\end{prop}
\begin{proof}
We give a proof by contradiction.
We assume that such $\Sigma$ exists.
Let $\eta \in C^{\infty}([0,+\infty))$ be a function such that $\eta \equiv 1$ on $[0,1]$,
$\eta \equiv 0$ on $[2, + \infty)$ and $|\eta'| \leq 2$.
Let $o\in \Sigma$.
For each integer $i\geq 1$,
we define
\begin{equation*}
    \phi_i(x) := \eta \left( \frac{d^{\Sigma}_o(x)}{i} \right).
\end{equation*}
By virtue of Lemma \ref{lem:key},
for every sufficiently large $i$,
\begin{align*}
0 &\leq \int_\Sigma\,  \left\{|\nabla_\Sigma \phi_i|^2- \left( \Ric_f^1(\nu,\nu)+ \left| \second_\Sigma  -\frac{f_\nu}{n-1}g_{\Sigma} \right|^2 \right)\phi^2_i \right\}\,\d m_{f,\Sigma}\\ \notag
&\leq \int_{B_{2i}^{\Sigma}(o)\backslash B_{i}^{\Sigma}(o)} |\nabla_{\Sigma}\, \phi_i|^2 \ \d m_{f,\Sigma} - \int_{B_{2i}^{\Sigma}(o)} \Ric_f^1(\nu,\nu)\phi_i^2\ \d m_{f,\Sigma}\\ \notag
    &\leq m_{f,\Sigma}(B_{2i}^{\Sigma}(o)\backslash B_{i}^{\Sigma}(o))  - \int_{B_{2}^{\Sigma}(o)} \Ric_f^1(\nu,\nu)\ \d m_{f,\Sigma}.
\end{align*}
By the finiteness of volume,
we see $m_{f,\Sigma}(B_{2i}^{\Sigma}(o)\backslash B_{i}^{\Sigma}(o)) \to 0$ as $i\to +\infty$,
and hence
\begin{equation*}
\int_{B_{2}^{\Sigma}(o)} \Ric_f^1(\nu,\nu)\ \d m_{f,\Sigma}\leq 0.
\end{equation*}
This contradicts with the positivity of $\Ric^1_f$.
We complete the proof.
\end{proof}

Impera-Rimoldi \cite{IR1} have presented non-existence results for stable $f$-minimal hypersurfaces satisfying volume growth conditions under the positivity of $\infty$-weighted Ricci curvature (see \cite[Theorems A and B]{IR1}).
We show similar results for $1$-weighted Ricci curvature.
\begin{prop}\label{prop:IR}
Let $(M,g,f)$ be a complete weighted manifold.
For $K > 0$,
we assume $\Ric_f^1 \geq Kg$.
Then there is no complete two-sided immersed stable $f$-minimal hypersurface in $M$ with $m_{f,\Sigma}(B^\Sigma_r(o)) = O(\e^{\alpha r})$ as $r\rightarrow + \infty$ for $o\in \Sigma$ and $\alpha \in (0, 2\sqrt{K})$.
\end{prop}
\begin{proof}
The proof is done by contradiction.
We possess the following relation of Brooks type between the bottom of the spectrum $\lambda_{1,\Sigma}$ of $\Delta_{f,\Sigma}$ and the volume entropy (cf. \cite[Theorem 1]{Brooks}, \cite[Theorem A]{IR1}):
\begin{equation*}
 \lambda_{1,\Sigma} \leq \frac{1}{4}\left( \limsup_{r\rightarrow + \infty} \frac{\log m_{f,\Sigma}(B^\Sigma_r(o))}{r} \right)^2. 
\end{equation*}
This together with the volume growth assumption leads us to
\begin{equation}\label{eq:Brooks}
\inf_{\phi\in C_0^\infty(\Sigma)\setminus \{0\}} \frac{\displaystyle \int_\Sigma |\nabla_\Sigma \,\phi|^2\ \d m_{f,\Sigma}}{\displaystyle \int_\Sigma \phi^2 \ \d m_{f,\Sigma}}\leq \lambda_{1,\Sigma}\leq \frac{\alpha^2}{4}.
\end{equation}
On the other hand,
the stability inequality \eqref{eq:key} yields
\begin{align*}
K \int_\Sigma \phi^2 \ \d m_{f,\Sigma}\leq \int_\Sigma \left(  \Ric_f^1(\nu,\nu) + \left| \second_\Sigma - \frac{f_\nu}{n-1} g_\Sigma\right|^2  \right)\phi^2 \ \d m_{f,\Sigma}\leq \int_\Sigma |\nabla_\Sigma \,\phi|^2\ \d m_{f,\Sigma}
\end{align*}
for every $\phi \in C^\infty_0(\Sigma)$;
in particular,
the left hand side of \eqref{eq:Brooks} is bounded from below by $K$.
This contradicts with $\alpha \in (0, 2\sqrt{K})$.
\end{proof}

We also have the following:
\begin{prop}\label{prop:IR2}
Let $(M,g,f)$ be an $n$-dimensional complete weighted manifold.
For $K > 0$,
we assume $\Ric_f^1 \geq K g$.
Then there is no complete two-sided immersed $f$-minimal hypersurface $\Sigma$ in $M$ with $\Ind_f(\Sigma)<+\infty$ such that one of the following holds:
\begin{enumerate}\setlength{\itemsep}{+0.7mm}
 \item[$(1)$] $m_{f,\Sigma}(\Sigma) = + \infty$ and $m_{f,\Sigma}(B^\Sigma_r(o)) \leq Cr^\alpha$ for any $r\geq r_0$ and some positive constants $C,r_0$ and $\alpha$;
\item[$(2)$] it holds that
\begin{equation}\label{eq:Main21}
m_{f,\partial B^{\Sigma}_r(o)}(\partial B^{\Sigma}_r(o))^{-1} \notin L^1(+\infty),\quad \left| \second_\Sigma - \frac{f_\nu}{n-1} g_\Sigma\right| \notin L^2(\Sigma,m_{f,\Sigma}).
\end{equation}
\end{enumerate}
\end{prop}
\begin{proof}
We give a proof by contradiction.
If $(1)$ holds,
then \cite[Proposition 6]{IR1} together with Lemma \ref{lem:key} implies
\begin{equation*}
    0 \geq \inf_{\Sigma\backslash B^\Sigma_r(o)} \left( \Ric_f^1 (\nu,\nu) + \left| \second_\Sigma - \frac{f_\nu}{n-1} g_\Sigma\right|^2 \right) \geq K
\end{equation*}
for every $r\geq r_0$,
and this contradicts with $K>0$.
On the other hand,
if $(2)$ holds,
then the second condition in \eqref{eq:Main21} and Lemma \ref{lem:key} yield
\begin{equation*}
\Ric_f^\infty(\nu,\nu) + \left| \second_\Sigma \right|^2 = \Ric_f^1(\nu,\nu) + \left| \second_\Sigma - \frac{f_\nu}{n-1}g_\Sigma \right|^2 \notin L^1(\Sigma,m_{f,\Sigma}).
\end{equation*}
By the first condition in \eqref{eq:Main21} and \cite[Proposition 7]{IR1},
for every compact subset $\Omega$ in $\Sigma$,
the bottom of the spectrum of the stability operator $\mathcal{L}_{f,\Sigma}$ over $\Sigma\setminus \Omega$ is negative.
This contradicts with the finiteness of $f$-index and \cite[Proposition 5]{IR1}.
\end{proof}

\subsection{Smooth compactness}\label{subsec:compactness}
Choi-Schoen \cite{CS} have established a smooth compactness of closed minimal surfaces in a compact three-dimensional manifold of positive Ricci curvature.
Sharp \cite{Sh} has extended their smooth compactness to the higher dimensional case.
In the weighted case,
such smooth compactness have been investigated by Li-Wei \cite{LW}, Cheng-Mejia-Zhou \cite{CMZ1}, \cite{CMZ2} and Barbosa-Sharp-Wei \cite{BSW} under the positivity of $\infty$-weighted Ricci curvature (see also \cite{CM1}, \cite{DX} for self-shrinker).
We prove smooth compactness under the positivity of $1$-weighted Ricci curvature.
Our main result is the following (cf. \cite{CS}, \cite[Theorem 1]{LW}):
\begin{thm}\label{thm:LW}
Let $(M,g,f)$ be a three dimensional closed weighted manifold of positive $1$-weighted Ricci curvature.
Then the space of smooth closed embedded $f$-minimal surfaces in $M$ of fixed topology is compact in the smooth topology.
\end{thm}

To prove it,
we first present the following result of Sharp type (cf. \cite{Sh}, \cite[Theorem 1.3]{BSW}):
\begin{prop}\label{prop:BWS}
For $3 \leq n \leq 7$,
let $(M,g,f)$ be an $n$-dimensional closed weighted manifold of positive $1$-weighted Ricci curvature.
For $\Lambda>0$ and a non-negative integer $I$,
let $\mathcal{S}_{\Lambda,I}$ be the space of smooth closed embedded $f$-minimal hypersurfaces $\Sigma$ in $M$ with 
\begin{equation*}
m_{f,\Sigma}(\Sigma)\leq \Lambda,\quad     \Ind_f(\Sigma) \leq I.
\end{equation*}
Then $\mathcal{S}_{\Lambda,I}$ is compact in the smooth topology.
\end{prop}
\begin{proof}
By Proposition \ref{prop:BM},
the universal cover of $M$ is compact;
in particular,
$M$ has finite fundamental group (see also Remark \ref{rem:closed}).
By passing to the universal cover,
we may assume that $M$ is simply connected.
Notice that
each hypersurface in $\mathcal{S}_{\Lambda,I}$ is two-sided (cf. \cite[Proposition 8]{CMZ2}). 
With the help of \cite[Proposition 3.3]{BSW},
every sequence $\{\Sigma_i\}_i$ in $\mathcal{S}_{\Lambda,I}$ has a convergent subsequence to a smooth closed embedded $f$-minimal hypersurface $\Sigma$ such that
the convergence is smooth (possibly with finite multiplicity) away from a locally finite collection of points in the limit $\Sigma$. 
To show that the convergence is smooth everywhere,
it suffices to verify that the multiplicity of the convergence is equal to one in view of the Allard regularity theorem.
If it is greater than one,
then $\Sigma$ is stable according to the argument of the proof of \cite[Theorem 1.3]{BSW} (cf. \cite[Proposition 3.2]{CM1}, \cite[Proposition 11]{CMZ2}).
This contradicts with Proposition \ref{prop:non-existence},
and hence the convergence is smooth everywhere.
Also,
by the same argument in \cite[Theorem 1.3]{BSW},
we have $\Ind_f(\Sigma) \leq I$.
Thus,
we complete the proof.
\end{proof}

Due to Proposition \ref{prop:BWS},
it suffices to derive index and volume estimates under the positivity of 1-weighted Ricci curvature in the three dimensional case.
For the index,
we possess the following estimate of Ejiri-Micallef type (see \cite{EM}, \cite[Proposition 1.8]{BSW}):
\begin{prop}[\cite{EM}, \cite{BSW}]\label{prop:EM}
Let $(M,g,f)$ be a three dimensional closed weighted manifold.
Let $\Sigma$ be a closed embedded $f$-minimal surface in $M$ of genus $\gamma$.
Then we have
\begin{equation*}
\Ind_f(\Sigma)\leq C_{M,f}\left( m_{f,\Sigma}(\Sigma)+\gamma-1 \right).
\end{equation*}
\end{prop}

We now aim to obtain volume estimates.
We deduce such estimates from a lower bound of the first eigenvalue of the Laplacian induced from an affine connection.
On an $n$-dimensional weighted manifold $(M,g,f)$,
we consider an affine connection
\begin{equation}\label{eq:connection}
\D_XY := \nabla_X Y -\frac{\d f(X)}{n-1}Y - \frac{\d f(Y)}{n-1}X.
\end{equation}
It has been observed that
the Ricci tensor induced from this affine connection coincides with $\Ric^1_f$ (see \cite[Proposition 3.3]{WY}, \cite{LX}).
Li-Xia \cite{LX} have studied a family of affine connections including \eqref{eq:connection},
and obtained an integral formula of Reilly type.
In a special case for \eqref{eq:connection},
it can be stated as follows (see \cite[Theorem 1.1]{LX}, and also \cite[Proposition 2.5, Remark 2.6]{FS}):
\begin{prop}[\cite{LX}]\label{prop:1-reilly}
Let $(M,g,f)$ be an $n$-dimensional compact weighted manifold with boundary.
Then for every $\varphi \in C^{\infty}(M)$,
\begin{align*}
 &\quad \,\,\int_M  \left\{ \left(\Delta_{\frac{n}{n-1}f} \,\varphi\right)^2 - \Ric_f^1\left( \nabla \varphi, \nabla \varphi \right) - \left| \mathrm{Hess}\, \varphi - \frac{1}{n-1}\langle \nabla\varphi,\nabla f \rangle g\right|^2 \right\}\ \d m_f\\
    &= \int_{\partial M} \left(\varphi_\nu^{\,2} H_{f,\partial M} +  2 \varphi_\nu \Delta_{f, \partial M}\,  \psi  + \mathrm{II}_{\partial M}\left(\nabla_{\partial M} \,\psi, \nabla_{\partial M}\, \psi \right)  \right) \ \d m_{f,\partial M},
\end{align*}
where $\psi := \varphi|_{\partial M}$,
and $\nu$ is the outer unit normal vector field on $\partial M$.
\end{prop}

In order to formulate the integral formula for general affine connections,
Li-Xia \cite{LX} have introduced the notion of the \textit{affine Laplacian} (see \cite[Definition 2.4]{LX}).
On a closed embedded hypersurface $\Sigma$ in an $n$-dimensional closed weighted manifold $(M,g,f)$,
the affine Laplacian associated with \eqref{eq:connection} is given by
\begin{equation*}
\Delta^\D_\Sigma:=\e^{\frac{f}{n-1}}\,\Delta_{f,\Sigma}.  
\end{equation*}
By direct calculations,
it coincides with the weighted Laplacian over a weighted manifold
\begin{equation*}
\left(\Sigma,\e^{-\frac{f}{n-1}}\,g_\Sigma, \frac{n+1}{2(n-1)}f\right).
\end{equation*}
In particular,
it is self-adjoint with respect to the weighted measure $m_{\frac{n}{n-1}f,\Sigma}$,
and its first non-zero eigenvalue has the following variational characterization:
\begin{equation}\label{eq:variation}
\lambda^\D_{1,\Sigma}=\inf_{\phi}\,\frac{\displaystyle \int_{\Sigma}\,|\nabla_\Sigma\, \phi|^2\,\d m_{f,\Sigma}}{\displaystyle \int_{\Sigma}\,\phi^2\,\d m_{\frac{n}{n-1}f,\Sigma}},
\end{equation}
where the infimum is taken over all $\phi \in C^{\infty}(\Sigma)\setminus \{0\}$ with $\int_{\Sigma}\,\phi\,\d m_{\frac{n}{n-1}f,\Sigma}=0$.
Furthermore,
we consider a conformally deformed metric
\begin{equation}\label{eq:conformal}
g_{f}:=\e^{-\frac{2n}{(n-1)^2}f}g
\end{equation}
over $M$.
Notice that
its volume measure $v_{g_{f},\Sigma}$ over $\Sigma$ coincides with $m_{\frac{n}{n-1}f,\Sigma}$,
and therefore
\begin{equation}\label{eq:conformal_relation}
\lambda^\D_{1,\Sigma}=\inf_{\phi}\,\frac{\displaystyle \int_{\Sigma}\,|\nabla_\Sigma\, \phi|^2_{g_f}\,\e^{-\frac{n+1}{(n-1)^2}f}\,\d v_{g_{f},\Sigma}}{\displaystyle \int_{\Sigma}\,\phi^2\,\d v_{g_{f},\Sigma}}
\leq \e^{-\frac{n+1}{(n-1)^2}\inf_\Sigma f}\lambda^{g_f}_{1,\Sigma}
\end{equation}
by \eqref{eq:variation},
where $\lambda^{g_f}_{1,\Sigma}$ is the first non-zero eigenvalue of the Laplacian on $\Sigma$ induced from $g_f$.

We obtain the following eigenvalue estimate of Choi-Wang type (cf. \cite{CW}, \cite[Theorem 2]{CMZ1}, \cite[Theorem 7]{LW}, \cite[Theorem 3]{MD}, \cite[Theorem 1.1]{F2}):
\begin{prop}\label{prop:CW}
Let $(M,g,f)$ be an $n$-dimensional closed weighted manifold.
For $K > 0$,
we assume $\Ric_f^1 \geq K\,\e^{-\frac{f}{n-1}}\,g$.
Let $\Sigma$ be a closed embedded $f$-minimal hypersurface in $M$.
Then we have
\begin{equation}\label{eq:CW}
    \lambda^\D_{1,\Sigma} \geq \frac{K}{2}.
\end{equation}
\end{prop}
\begin{proof}
The desired estimate has been  proved by the first named author \cite{F2} for general affine connections (see \cite[Theorem 1.1]{F2}).
For the completeness, we here give its proof.
Proposition \ref{prop:BM} tells us that
the universal cover of $M$ is compact;
in particular,
$M$ has finite fundamental group (see also Remark \ref{rem:closed}).

We first show the estimate when $M$ is simply connected.
Then $M$ and $\Sigma$ are orientable since $\Sigma$ is closed and embedded.
Further,
by virtue of Proposition \ref{prop:Frankel},
$\Sigma$ must be connected.
Therefore,
it divides $M$ into two components $U_1$ and $U_2$ (see e.g., \cite[Lemma 6]{LW}).
Let $u$ be an eigenfunction for $\lambda^\D_{1,\Sigma}$.
We may assume
\begin{equation*}
    \int_{\partial U_1} \second_{\partial U_1}(\nabla_{\partial U_1}\, u,\nabla_{\partial U_1}\, u)\ \d m_{f,\partial U_1} \geq 0.
\end{equation*}
Let $\varphi$ be the solution to the following boundary value problem:
\begin{equation*}
    \begin{cases}
        \Delta_{\frac{n}{n-1}f} \,\varphi = 0 & \mbox{ on }U_1,\\
        \varphi = u & \mbox{ on }\partial U_1.
    \end{cases}
\end{equation*}
By using Proposition \ref{prop:1-reilly} over $U_1$, we obtain
\begin{align*}\label{eq:Choi-Wang-1}
    0&\geq K \int_{U_1} |\nabla \varphi|^2 \ \d m_{\frac{n}{n-1}f} - 2\lambda^\D_{1,\Sigma}\int_{\partial U_1} \varphi_\nu \, u \ \d m_{\frac{n}{n-1}f,\partial U_1} + \int_{\partial U_1} \second_{\partial U_1}(\nabla_{\partial U_1} \, u, \nabla_{\partial U_1} \, u)\ \d m_{f,\partial U_1}\\ \notag
    &\geq (K- 2\lambda^\D_{1,\Sigma}) \int_{U_1} |\nabla\varphi|^2 \ \d m_{\frac{n}{n-1}f},
\end{align*}
where $\nu$ denotes the outer unit normal vector field on $\partial U_1$.
This leads us to \eqref{eq:CW}.

For the general case,
in view of the finiteness of the fundamental group,
the lift $\overline{\Sigma}$ of $\Sigma$ is a closed embedded $\bar{f}$-minimal hypersurface in the universal cover,
where $\bar{f}$ is the lift of $f$.
The result in the above paragraph says that
the first non-zero eigenvalue $\lambda^\D_{1,\overline{\Sigma}}$ of $\Delta^\D_{\overline{\Sigma}}$ satisfies
\begin{equation*}
\lambda^\D_{1,\Sigma}\geq \lambda^\D_{1,\overline{\Sigma}} \geq \frac{K}{2}.
\end{equation*}
This completes the proof.
\end{proof}

We arrive at the following volume estimate of Yang-Yau type (cf. \cite{YY}, \cite[Proposition 8]{CMZ1}):
\begin{prop}\label{prop:YY}
Let $(M,g,f)$ be a three dimensional closed weighted manifold. 
For $K > 0$,
we assume $\Ric_f^1 \geq K\,\e^{-\frac{f}{2}}\,g$.
Let $\Sigma$ be a closed embedded $f$-minimal hypersurface in $M$.
Then
\begin{equation*}
 m_{f,\Sigma}(\Sigma)\leq \frac{16\pi}{K}\left(     \frac{2}{|\pi_1(M)|}-\frac{1}{2}\chi(\Sigma)    \right)\,\e^{-\inf_{\Sigma}f+\frac{1}{2}\sup_\Sigma f},
\end{equation*}
where $\chi(\Sigma)$ denotes the Euler characteristic of $\Sigma$.
\end{prop}
\begin{proof}
Similarly to the proof of Proposition \ref{prop:CW},
we first prove the estimate when $M$ is simply connected.
Then $\Sigma$ is orientable,
and we denote its genus by $\gamma$. 
We consider the conformally deformed metric $g_f$ defined as \eqref{eq:conformal}.
Similarly to the proof of \cite[Proposition 7]{CMZ1},
applying the classical Yang-Yau estimate to the metric $g_f$ tells us that
\begin{equation}\label{eq:YY1}
\e^{-\frac{1}{2}\sup_\Sigma f} m_{f,\Sigma}(\Sigma) \leq m_{\frac{3}{2}f,\Sigma}(\Sigma)=v_{g_f,\Sigma}(\Sigma)\leq \frac{8\pi(\gamma+1)}{\lambda^{g_f}_{1,\Sigma}}.
\end{equation}
On the other hand,
by combining \eqref{eq:conformal_relation} and Proposition \ref{prop:CW},
we possess
\begin{equation}\label{eq:YY2}
\e^{-\inf_\Sigma f}\lambda^{g_f}_{1,\Sigma} \geq \lambda^\D_{1,\Sigma}\geq \frac{K}{2}. 
\end{equation}
From \eqref{eq:YY1} and \eqref{eq:YY2},
it follows that
\begin{equation}\label{eq:YY3}
\e^{-\frac{1}{2}\sup_\Sigma f} m_{f,\Sigma}(\Sigma)\leq  \frac{16\pi}{K}(\gamma+1)\,\e^{-\inf_\Sigma f}.
\end{equation}
This proves the desired one.
For the general case,
we complete the proof by applying \eqref{eq:YY3} to the universal cover similarly to the proof of \cite[Proposition 8]{CMZ1}.
\end{proof}

We are now in a position to conclude Theorem \ref{thm:LW}.
\begin{proof}[Proof of Theorem \ref{thm:LW}]
It immediately follows from Propositions \ref{prop:BWS}, \ref{prop:EM} and \ref{prop:YY}.
\end{proof}

\subsection*{{\rm{Acknowledgements}}}
The authors are grateful to Professor Keita Kunikawa for fruitful discussions.
The authors also thank Professors Mattia Fogagnolo, Shin-ichi Ohta and Chao Xia for their comments.
The first named author was supported by JSPS KAKENHI (25KJ0271).
The second named author was supported by JSPS KAKENHI (22H04942, 23K12967).

\end{document}